\documentclass[11pt]{article}
\usepackage{amssymb, latexsym, mathtools, amsthm, amsmath}
\usepackage{amsfonts}
\usepackage{stmaryrd}
\usepackage{authblk}
\renewenvironment{abstract}
 { \normalsize
  \list{}{\setlength{\leftmargin}{.0cm}%
    \setlength{\rightmargin}{\leftmargin}}%
  \item {\bf \abstractname.}\relax}
 {\endlist}
\usepackage{xcolor,colortbl}
\usepackage{enumerate}
\usepackage{pdfsync}
\usepackage[numbers]{natbib}
\usepackage{parskip}
\theoremstyle{plain}

\newenvironment{remark}{{\bf Remark}.}{\hfill$\blacktriangleleft$}
\newtheorem{thm}{Theorem}[section]

\newtheorem{lem}[thm]{Lemma}

\newtheorem{coro}[thm]{Corollary}

\theoremstyle{definition}
\newtheorem{defi}[thm]{Definition}
\newcommand{\twolo}{2^{<\omega}}

\newcommand{\Nat}{\mathbb{N}}

\newcommand{\ds}{\textup{\textsf{d}}}

\newcommand{\restr}{\upharpoonright}  

\newcommand{\un}{\uparrow} 
\newcommand{\de}{\downarrow} 

\DeclarePairedDelimiter{\tuple}{\langle}{\rangle}

\newcommand{\sqbrad}[2]{\{\hspace{0.03cm}{#1} : {#2}\hspace{0.03cm}\}}

\DeclarePairedDelimiter{\dbra}{\llbracket}{\rrbracket}

\newcommand{\DNC}{\textsf{\textup{DNC}}}

\newcommand{\CC}{\mathcal{C}}

 \usepackage{tcolorbox}

\newcommand{\abs}[1]{|{#1}|}

\newcommand{\parb}[1]{\big({#1}\big)}
\newcommand{\parB}[1]{\Big(\hspace{0.04cm}{#1}\hspace{0.04cm}\Big)}

\newcommand{\dom}{\mathsf{dom}}

\newcommand{\ml}{Martin-L\"{o}f }

\newcommand{\pz}{\Pi^0_1}

\newcommand{\pzt}{\Pi^0_2}

\newcommand{\szth}{\Sigma^0_3}
\newcommand{\eg}{e.g.\ }

\newcommand{\ce}{c.e.\ }

\newcommand{\pf}{prefix-free }

\newcommand{\aed}{a.e.\ dominating\ }

\newcommand{\twome}{2^{\omega}}

\newcommand{\zj}{\emptyset'}
\newcommand{\twomel}{2^{<\omega}}

\newcommand{\wedga}{\ \wedge\ }

\newcommand{\leqT}{\leq_T}
\newcommand{\geqT}{\geq_T}
\newcommand{\equivT}{\equiv_T}

\newcommand{\impl}{\implies}

\newcommand{\hthree}{\hspace{0.3cm}}

\newcommand{\Ss}{\mathsf{S}}

\newcommand{\inv}{^{-1}}

\newcommand{\leqlr}{\leq_{LR}}
\newcommand{\geqlr}{\geq_{LR}}

\title{Collision-resistant hash-shuffles on the reals\thanks{Supported by Beijing Natural Science Foundation (IS24013).}}
\author{George Barmpalias} \author{Xiaoyan Zhang\thanks{Authors are in alphabetical order. We thank L.~Levin for several suggestions.} } 
\affil{State Key Lab of Computer Science, Institute of Software\\ Chinese Academy of Sciences, Beijing, China}

\begin{document}
\maketitle
\begin{abstract}
Oneway real functions  are 
effective maps on positive-measure sets of reals that 
preserve randomness and have no effective probabilistic inversions. 
We construct a oneway real function 
which is {\em collision-resistant}: the probability of effectively producing 
distinct reals with the same image is zero, and each real has uncountable inverse image.
\end{abstract}
\section{Introduction}
Oneway functions underly much of the theory of computational complexity  \cite{Levin2003}: 
they are finite maps that are
computationally easy to compute but  hard to invert, even probabilistically.
Modern cryptographic primitives rely on their existence, an unproven hypothesis which
remains  a long-standing open problem.
Non-injective  oneway functions  play a special role in
public-key cryptography, especially when  they are {\em collision-resistant}: no algorithm can generate {\em siblings} (inputs with the same output) with positive probability in a resource-bounded setting \cite{manytooneCrypto}.

Levin \cite{LevinEmailDec23} extended this concept  to computable functions on the {\em reals} (infinite binary sequences)
in the  framework of computability theory and algorithmic randomness \cite{rodenisbook, MR1438307}.
They are partial computable real functions that preserve randomness in the sense of \citet{MR0223179}
and  no  probabilistic algorithm inverts them with positive probability.

A total computable oneway surjection $f$ was constructed in \cite{onewayg24} via 
a partial permutation of the bits of the input based on an effective enumeration of 
the halting problem $\zj$. Independently \citet{Gacs24oneway} constructed a partial computable  function 
which is probabilistically hard to invert in a different setting, where probability is  over the domain 
rather than the range.

Given the importance of collision-resistant oneway functions in computational complexity
\citet{LevinEmailJul24} asked whether collision-resistant computable oneway real functions
exist. Although the oneway function in \cite{onewayg24} 
is  strongly nowhere injective (all inverse images are uncountable)
but not {\em collision-resistant}: a Turing machine can produce {\em $f$-siblings}
(reals $x\neq z$ with $f(x)=f(z)$) given any sufficiently algorithmically random oracle. 
So $f$-siblings can be effectively produced by a probabilistic  machine with positive probability.

Our goal  is to establish 
the existence of a total computable nowhere injective collision-resistant oneway function.
The key idea is apply a {\em hash} to the partial permutation ({\em shuffle}) used in
the original oneway function \cite{onewayg24} with a boolean function $h$ 
and show that  (under mild assumptions on $h$) these {\em hash-shuffles}
are also oneway and (strongly) nowhere injective. We then define a specific $h$ based 
on the universal  partial computable predicate and show that corresponding 
hash-shuffle is collision-resistant.

Given that oneway permutations \cite{onewayPermImp, onewaypermRothe}
are also significant in computational complexity, it is interesting to know whether
injective oneway  maps on the reals exist.
This is not known but by \cite[Corollary 3.2]{onewayg24} they cannot be total computable.
Assuming random-preservation we show that inverting partial computable  injections
is in general easier than inverting total computable many-to-one maps on the reals.

{\bf Outline.}
Oneway functions and collision-resistance are defined in
\S\ref{4okhjPbTWo},
where we also show that the {\em shuffles} of \cite{onewayg24} are not collision-resistant.

Hash-maps and their corresponding 
 {\em hash-shuffles} are defined in
\S\ref{opCuMKHY9j} and shown to be oneway under mild assumptions on their hash-map.
 This analysis also shows how to obtain oneway functions of different strengths, in terms of the Turing degrees of the
 oracles that can probabilistically invert them.

A collision-resistant oneway function is obtained in \S\ref{tEVOhEWHi7} 
by specifying an appropriate hash function based on a universal Turing machine.

We conclude in \S\ref{952rgaHVWk} by establishing an upper bound on the hardness of 
partial computable oneway injections which is lower than the worse-case for total computable oneway
many-to-one maps on the reals.

{\bf Notation.}
Let $\Nat$ be the set of natural numbers, represented by $n,m, i,j,t,s$.
Let $2^\omega$ be the set of  reals, represented by variables $x,y,z,v,w$, and $2^{<\omega}$ the set of  strings which we represent by $\sigma,\tau, \rho$. We index the bits $x(i)$ of $x$ starting from $i=0$. The prefix of $x$ of length $n$  is $x(0)x(1)\cdots x(n-1)$
and is  denoted by $x\restr_n$. Let
$\preceq$, $\prec$ denote the prefix and strict prefix relation between two strings or a string and a real. 
Similarly $\succeq, \succ$ denote the suffix relations.
Let $x\oplus y$ denote the real $z$ with $z(2n)=x(n)$ and $z(2n+1)=y(n)$.

The \textit{Cantor space} is $2^\omega$  with the topology generated by the 
basic open sets
\[
\dbra{\sigma}:=\sqbrad{z\in\twome}{\sigma\prec z}
\hspace{0.3cm}\textrm{for $\sigma\in 2^{<\omega}$.}
\]
Let $\mu$ be the  \textit{uniform measure} on $\twome$, determined by $\mu(\dbra{\sigma})=2^{-|\sigma|}$.  
Probability in $\twome\times\twome$ is reduced
to $\twome$ via  $(x,y)\mapsto x\oplus y$. 
A subset of $\twome$ is {\em positive} if it has positive $\mu$-measure and {\em null} otherwise. Let
\begin{itemize}
\item $\de, \un$ denote that the preceding expression is defined or undefined
\item $f:\subseteq\twome\to\twome$ denote that $f$ is a function from a subset of $\twome$ to $\twome$
\item $\dom(f)$ be the {\em domain} of $f$: the set of $x\in\twome$ where $f(x)$ is defined.
\end{itemize}
Turing reducibility $x\leqT z$ means that $x$ is computable from $z$ (is {\em $z$-computable}). 
Effectively open sets or $\Sigma^0_1$ classes are of the form $\bigcup_i\dbra{\sigma_i}$ where $(\sigma_i)$ is computable. 
A family $(V_n)$ is called uniformly $\Sigma^0_1$ if
\[
\textrm{$V_n=\bigcup_i\dbra{\sigma_{n,i}}$ \ \ where $(\sigma_{n,i})$ is computable.}
\]
A \textit{Martin-L\"of test} is a uniformly $\Sigma^0_1$ sequence $(V_n)$ such that $\mu(V_n)\leq 2^{-n}$. A real $x$ is \textit{random} if 
$x\notin\bigcap_n V_n$ for any Martin-L\"of test $(V_n)$. 
Relativization to oracle $r$ defines $\Sigma^0_1(r)$ classes and $r$-random reals.

\section{Oneway functions and collisions}\label{4okhjPbTWo}
Oneway functions where introduced in \cite{DHellmanow, purdyoneway}.
Levin \cite{LevinEmailDec23} adapted this notion to effective maps on the reals.
Let $f, g:\subseteq\twome\to\twome$. 

We say that  $g$ is a {\em probabilistic inversion} of $f$ if 
\[
\mu(\sqbrad{y\oplus r}{f(g(y\oplus r))=y})>0
\]
and say that $f$ is {\em random-preserving} if $\mu(\dom(f))>0$ and 
\[
\textrm{$f(x)$ is random for each random $x\in\dom(f)$.}
\]
These are the ingredients of  Levin's definition of  oneway real functions.
\begin{defi}[Levin]\label{jpr3eHEJ4levin}
We say that  $f\subseteq: \twome\to\twome$  is {\em oneway} if it
\begin{itemize}
\item  is partial computable and  random-preserving
\item  has no partial computable probabilistic inversion.
\end{itemize}
If $f$ has no probabilistic inversion $g\leqT w$ it is {\em oneway relative to $w$.}
\end{defi}
\begin{remark}
Oneway functions can be defined with `randomness-preserving' replaced with the 
weaker condition that with positive probability $f$ maps to random reals.
It is not hard to show that the two formulations are essentially equivalent, up to effective restrictions \cite[Lemma 3.5]{owrand24}.
\end{remark}

Let $(a_i)$ be an effective enumeration of $\zj$ without repetitions. 

By \cite[Theorem 4.4]{onewayg24} the total computable function 
\begin{equation}\label{HEBpofpqIB}
f:\twome\to\twome
\hspace{0.3cm}\textrm{given by}\hspace{0.3cm}
f(x)(i):=x(a_i)
\end{equation}
is a oneway surjection. 
By \cite[Theorem 4.9]{onewayg24} 
\[
\textrm{$f\inv(y)$ is uncountable for each $y\in f(\twome)$.}
\]
Unfortunately $f$ lacks the desired property of collision-resistance.

The notion of \citet{neglivyuginold}  of {\em negligibility} (also see \cite{bslBienvenuP16}) is handy.
\begin{defi}[V'yugin]
A class $\CC\subseteq\twome$ is {\em negligible}
if the  set of oracles that compute a member of $\CC$ is null.
If the set of oracles $z$ such that $w\oplus z$ computes a member of $\CC$ is null
we say that $\CC$ is {\em $w$-negligible}.
\end{defi}
\citet{LevinEmailJul24} defined collision-resistance for real functions.
\begin{defi}[Levin]
Given $f:\subseteq \twome\to\twome$ the members of 
\[
S_f:=\sqbrad{(x,z)}{x\neq z\wedga f(x)=f(z)}
\]
are called {\em $f$-siblings}. We say that $f$ is {\em collision-resistant} if 
$S_f$ is negligible and {\em collision-resistant relative to $w$}
if $S_f$ is $w$-negligible.
\end{defi}
To see that $f$ of \eqref{HEBpofpqIB} is not collision-resistant note that it is a 
{\em shuffle}:  it outputs a permutation of selected bits of the input.
The selected positions are the members of $\zj$ so if we fix $k\not\in \zj$  and let
\[
\textrm{$x_k$ be the real $z$ with $\forall i\ \parb{z(i)=x(i)\iff i\neq k}$}
\]
then $x\mapsto (x, x_k)$ is computable and each output is an $f$-sibling.

Toward collision-resistance we could use another  \ce set $A$ in place of $\zj$ in its definition
which has {\em thin} infinite complement: the oracles that compute an infinite subset of $\Nat-A$ form a null class.
The existence of such $A$ is well- known (any hypersimple set has this property). 

With this modification $f$ would still fail collision-resistance  but
would satisfy the weaker property that the oracles computing members of 
\[
\sqbrad{(x,z)}{\forall i_0\ \exists i>i_0,\ x(i)\neq z(i) \wedga f(x)=f(z)}
\]
is null. Restricting $f$ to a positive subset of $\twome$ while keeping it partial computable does not
make $f$ collision-resistant. These attempts show that obtaining collision-resistant oneway functions 
requires a new ingredient.

\section{Hash shuffles}\label{opCuMKHY9j}
Toward achieving collision resistance we first extend the shuffle format \eqref{HEBpofpqIB}
for oneway functions $f$ by  effectively adding  some ``noise'' to the output of $f$ by combining it with the output
of another function $h$ which we call a {\em hash}.

\begin{defi}
If $A\subseteq\Nat$ is an infinite \ce set a computable 
\[
h: \sqbrad{\sigma}{|\sigma|\in A}\to \{0,1\}
\]
is called an {\em $A$-hash} or simply a {\em hash}.
\end{defi}
Let $\otimes$ denote the XOR operator between bits.

\begin{defi}
If $h$ is an $A$-hash the  {\em $h$-shuffle} is the 
\[
f:\twome\to\twome
\hspace{0.3cm}\textrm{given by}\hspace{0.3cm}
f(x)(i):=x(a_i)\otimes h(x\restr_{a_i})
\]
where $(a_i)$ is a computable enumeration  of $A$ without repetitions.
\end{defi}
Under mild assumptions on $h$ these  generalized shuffles preserve the properties of 
\eqref{HEBpofpqIB}. An $A$-hash $h$ is {\em trivial} if $A$ is computable  and
\[
\textrm{{\em hash-shuffles} refer to $h$-shuffles for non-trivial $h$. }
\]
We show that hash-shuffles $f$ are oneway and nowhere injective. Our analysis 
differs from \cite{onewayg24} and establishes  additional properties: (a)
there are oneway functions of different strengths according to the choice of the domain $A$  of the hash;
(b) no oracle $w\not\geqT A$ can invert $f$ on any random $y\not\geqT A$.

\subsection{Properties of hash-shuffles}\label{8ti2mZr8ag}
We establish the basic properties of hash-shuffles.
\begin{lem}\label{9IqHVDCOIU}
For each \ce set $A$ and $A$-hash $h$ the $h$-shuffle $f$ is 
\begin{itemize}
\item  total computable,  surjective and random-preserving
\item strongly nowhere injective: $f\inv(y)$ is uncountable for each $y$. 
\end{itemize}
\end{lem}\begin{proof}
Let $(a_i)$ be a computable enumeration  of $A$ without repetitions
and $f$ be the $h$-shuffle. Clearly 
$f$ is total computable and 
\begin{equation}\label{TpoGTf2D7V}
f(x)(i)\otimes h(x\restr_{a_i})=\parb{x(a_i)\otimes h(x\restr_{a_i})} \otimes h(x\restr_{a_i})=x(a_i).
\end{equation}
For each $y$ let $x$ be given by
\begin{equation}\label{liYIXCXVxA}
x(m):=
\begin{cases}
y(i)\otimes h(x\restr_{a_i})&\textrm{if $a_i=m$}\\[0.2cm]
\ \ \ \ \ \ 0&\textrm{otherwise.}
\end{cases}
\end{equation}
By \eqref{TpoGTf2D7V}  we have $f(x)=y$ so $f$ is a surjection. 
By replacing 0 in \eqref{liYIXCXVxA} with arbitrary bits we get that $f\inv(y)$ is a perfect set,
hence uncountable.

The oracle-use of $f(x)(n)$ is $\ell_n:=\max\sqbrad{a_i}{i\leq n}+1$. Let 
\[
V_{\tau} :=\ \sqbrad{\sigma\in 2^{\ell_{|\tau|}}}{f(\sigma)\preceq\tau}
\]
so $\dbra{V_{\tau}}=f\inv(\dbra{\tau})$.
Since $f(x)(n)$ depends exclusively on $x(a_n)$, $x\restr_{a_n}$:
\begin{enumerate}[(i)]
\item $\mu(V_{\tau i})=\mu(V_{\tau})/2$ for  $\tau\in\twomel$,  $i<2$, so $\mu(V_{\tau})=2^{-|\tau|}$
\item $\dbra{\tau}\cap\dbra{\rho}=\emptyset\impl \dbra{V_{\tau}}\cap \dbra{V_{\rho}}=\emptyset$
\item $V_i$ is finite and $i\mapsto V_i$ is computable.
\end{enumerate}
Let $(U_j)$ be a universal \ml test with \pf  $U_i\subseteq\twomel$ and
\[
E_j:=\bigcup_{\tau\in U_j} V_{\tau}.
\]
By (iii) the sets $E_j$ are \ce uniformly in $j$. By (i), (ii) we get
\[
\mu(E_j)=\sum_{\tau\in U_j} \mu(V_{\tau}) =\sum_{\tau\in U_j} 2^{-|\tau|}=\mu(U_j)\leq 2^{-j}
\]
so $(E_j)$ is a \ml test. Since $f\inv(\dbra{U_j})=\llbracket E_j\rrbracket$ if $y$ is not random and $f(x)=y$ then
$x$ is not random. So $f$ is random-preserving.
\end{proof}

\subsection{Inversions of hash-shuffles}\label{DZjVDbI3sf}
We show that  hash-shuffles have no computable probabilistic inversions.
\begin{defi}
A {\em prediction} is a partial  $p:\subseteq \twomel\to  \{0,1\}$ and 
\begin{itemize}
\item $y$ is {\em $p$-predictable} if $p(y\restr_n)\downarrow$ for infinitely many $n$ and \[p(y\restr_n)\downarrow\implies y(n)=p(y\restr_n)\]
\item $y$ is {\em $r$-predictable} if it is $p$-predictable for a prediction $p\leqT r$.
\end{itemize}
\end{defi}
We need a property of random reals regarding predictions.
\begin{lem}\label{EFUhx24qYD}
If $y$ is $r$-predictable then  $y$ is not $r$-random.
\end{lem}\begin{proof}
Without loss of generality assume that $r$ is computable.
Assuming that $y$ is $p$-predictable for a partial computable prediction $p$
it suffices to construct a \ml test $(V_i)$  with $y\in \bigcap_i \dbra{V_i}$.

Let $\hat V_0, V_0\subseteq\twomel$ be \ce  and \pf sets  such that
\begin{itemize}
\item $\hat V_0$ contains a prefix of every $p$-predictable real
\item $V_0:=\sqbrad{\sigma \ p(\sigma)}{\sigma\in \hat V_0}$.
\end{itemize}
Then $\mu(V_0)\leq 1/2$.
Assuming that $V_i$ has been defined let 
$\hat V_{i+1}, V_{i+1}\subseteq\twomel$ be \ce  and \pf  containing proper extensions of strings in $V_i$ and
\begin{itemize}
\item $\hat V_{i+1}$ contains a prefix of every $p$-predictable real
\item $V_{i+1}:=\sqbrad{\sigma\  p(\sigma)}{\sigma\in \hat V_{i+1}}$.
\end{itemize}
Then $(V_i)$ are uniformly \ce and
$\mu(V_{i+1})\leq \mu(\hat V_{i+1})/2\leq \mu(V_{i})/2$.

So $(V_i)$ is a \ml test and by definition $y\in \bigcap_i \dbra{V_i}$.
\end{proof}
We say that $\hat g$ is a {\em representation} of  $g: \subseteq \twome\to \twome$  if 
\begin{itemize}
\item $\hat g: \twomel\to \twomel$ is $\preceq$-preserving and $\hat g(\lambda)=\lambda$
\item $g(x)\de\iff  \lim_{\tau\prec x} \hat g(\tau)= g(x)\iff \lim_{\tau\prec x} \abs{\hat  g(\tau)}=\infty$.
\end{itemize}
Every partial computable $g$ has a computable representation $\hat g$. 
\begin{lem}\label{ty4nwYJGee}
If $h$ is an $A$-hash and $f$ is the $h$-shuffle 
\begin{enumerate}[(i)]
\item $f$ has an $A$-computable inversion
\item $f$ is not probabilistically invertible on any random $y\not\geqT A$
\item $f$ is oneway relative to each  $w\not\geqT A$.
\end{enumerate}
\end{lem}\begin{proof}
Let $(a_i)$ be a computable enumeration of $A$ without repetitions so 
\[
f(x)(i)=x(a_i)\otimes h(x\restr_{a_i})
\]
defines the $(A,h)$-shuffle.
Define $\ds:\twome\to\twome$ by
\[
\ds(y)(n)=
\begin{cases}
y(i)\otimes h(\ds(y)\restr_n)&\textrm{if $a_i=n$}\\[0.2cm]
\hspace{0.8cm}0&\textrm{if $n\not\in A$}
\end{cases}
\]
so $\ds\leqT A$. For each $i,n$ with  $a_i=n$ we have
\[
\ds(y)(a_i) = \ds(y)(a_i) = y(i)\otimes h(\ds(y)\restr_{n})
\]
so for $x:=\ds(y)$ and each $i$ we have
\begin{equation}\label{QufRiIdtk2}
f(\ds(y))(i)=\ds(y)(a_i)\otimes h(x\restr_{a_i}) = y(i).
\end{equation}
This implies $f(\ds(y))=y$ for each $y$ which concludes the proof of (i).

Assuming that $g$ is partial computable, $y\not\geqT A$ is random and 
\[
E:=\sqbrad{r}{f(g(y, r))=y}
\] 
it remains to show that $\mu(E)=0$. For a contradiction assume otherwise and let
$r\in E$ be such that $y$ is $r$-random and $A\not\leq_T y\oplus r$.
Let 
\[
g_r(z):=g(z,r)
\hspace{0.3cm}\textrm{so}\hspace{0.3cm}
g_r\leqT r
\hspace{0.3cm}\textrm{and}\hspace{0.3cm}
f(g_r(y))=y.
\]
We define a prediction $p\leqT r$. For each $i$ and $\sigma\in 2^i$ let
\[
p(\sigma):\simeq
\begin{cases}
\hat g_r(\sigma)(a_{i+1})\otimes h(\hat g_r(\sigma)\restr_{a_{i+1}})
&\textrm{if $|\hat g_r(\sigma)|>a_{i+1}$}\\[0.2cm]
\hspace{1.5cm}\un &\textrm{otherwise.}
\end{cases}
\]
where $\hat g_r\leqT r$ is a representation of $g_r$.
Since $f(g_r(y))=y$, for all $i$
\[
y(i+1)
=f(g_r(y))(i+1)=g_r(y)(a_{i+1})\otimes h(g_r(y)\restr_{a_{i+1}})
\] 
so if $p(\sigma)$ halts for $\sigma\prec y$ then $p$ predicts $y$ correctly on $\sigma$. 

Since $y$ is $r$-random, by Lemma \ref{EFUhx24qYD}, $p$ cannot predict $y$ infinitely often so 
$\exists j_0\ \forall i>j_0,\ p(y\restr_i)\uparrow$.
Since $\forall i,\ \hat g_r(y\restr_i)\downarrow$   we get 
\begin{equation}\label{8f4bd488}
|\hat g_r(y\restr_i)|\leq a_{i+1}\ \ \text{ for all }i>j_0.
\end{equation}
Let $t_n:=\min\sqbrad{t}{t>j_0\wedga |\hat g_r(y\restr_{t})|>n}$. We claim that
\[
n\in A\iff n\in\{a_0,\dots, a_{t_n}\}.
\]
Indeed if $n=a_{i+1}$ for some $i\geq t_n$ then 
\[
|\hat g_r(y\restr_i)|\geq |\hat g_r(y\restr_{t_n})|>n=a_{i+1}
\]
which contradicts  \eqref{8f4bd488}. Since $(t_n)\leqT y\oplus r$
we get  $A\leq_T y\oplus r$ which contradicts the choice of $r$.
We conclude that  $\mu(E)=0$ so (ii) holds. 

For (iii) consider the above argument for  $w\not\geqT A$ and $g\leqT w$.
If
\begin{equation}\label{Wp6nk3eFc}
y\oplus w\not\geqT A\wedga \textrm{$y$ is $w$-random}
\end{equation}
the above argument gives $\mu(E)=0$.
Since \eqref{Wp6nk3eFc} holds 
for almost all $y$  there is no probabilistic inversion $g\leqT w$ of $f$ so (iii) holds.
\end{proof}
\begin{coro}\label{QZG9TROwmy}
If $h$ is an $A$-hash then the $h$-shuffle is 
\begin{enumerate}[(i)]
\item total computable and random-preserving  
\item surjective and nowhere injective
\item oneway relative to  each $w\not\geqT A$.
\end{enumerate}
\end{coro}\begin{proof}
By Lemma \ref{9IqHVDCOIU} we get  (i), (ii) and by
Lemma \ref{ty4nwYJGee} we get (iii).
\end{proof}
\section{Collision-resistance}\label{tEVOhEWHi7}

We exhibit a total computable oneway and nowhere injective collision-resistant $f:\twome\to\twome$.
In \S\ref{ID1ntaD4Cc} we define a hash $h$ based on the universal enumeration 
and show that the $h$-shuffle has the required properties.

In \S\ref{epEBCzKPQM} we obtain hashes from  pairs of disjoint  \ce sets
and use them to obtain collision-resistant oneway functions of various strengths.
Let 
\[
\textrm{ $(\sigma,n)\mapsto\tuple{\sigma, n}$\ \ with\ \ $|\sigma|< \tuple{\sigma, n}$}
\]
be a computable bijection between $\twomel\times\Nat, \Nat$.

\subsection{Hashing with the universal predicate}\label{ID1ntaD4Cc}
Fix a universal effective enumeration $(\varphi_i)$  of all partial computable boolean functions on $\Nat$
so $\zj=\sqbrad{i}{\varphi_i(i)\de}$ is the halting set.

The {\em universal-predicate} is $\varphi_i(i)$ and
a boolean {\em total extension} of it is a boolean function $\psi$ with $\forall i\in\zj,\ \psi(i)=\varphi_i(i)$.

\begin{lem}\label{SkANwtV64A}
There is a hash-shuffle $f$ such that every pair of $f$-siblings computes a  boolean total extension of the universal predicate.
\end{lem}\begin{proof}
Let $(\sigma_i, n_i)$ be an effective enumeration of  $\twomel\times \zj$.

Let $A:=\sqbrad{\tuple{\sigma_i, n_i}}{i\in\Nat}$ and define the $A$-hash: 
\[
h(\tau):= 
\begin{cases}
\varphi_{n_i}(n_i)&\textrm{if $\sigma_i\prec \tau\wedga \tau\in 2^{\tuple{\sigma_i, n_i}}$}\\[0.2cm]
\ \ \ 0 &\textrm{if $\sigma_i\not\prec \tau\wedga \tau\in 2^{\tuple{\sigma_i, n_i}}$}
\end{cases}
\]
so the $h$-shuffle  is given by 
$f(x)(i)=:h(x\restr_{\tuple{\sigma_i, n_i}})\otimes x(\tuple{\sigma_i, n_i})$.

Suppose that $x,z$ are $f$-siblings and let 
$\sigma$ be the least prefix of $x$ which is not a prefix of $z$. 
By the definition of $f$ and $f(x)=f(z)$ 
\[
\forall n\in \zj\ \parb{x(\tuple{\sigma,n})=z(\tuple{\sigma,n}) \iff \varphi_n(n)=0}.
\]
Therefore $\psi(n):=x(\tuple{\sigma,n})\otimes z(\tuple{\sigma,n})$ is an $(x\oplus z)$-computable boolean total extension of $\varphi_i(i)$.
\end{proof}
We say  $\psi:\Nat\to\Nat$ is  {\em diagonally non-computable}  if $\forall i\in\zj, \psi(i)\neq \varphi_i(i)$. Let
\[
\DNC_2:=\sqbrad{\psi}{\forall n,\ \psi(n)\in\{0,1\}\wedga \forall i\in\zj, \psi(i)\neq \varphi_i(i)}
\]
be the set of diagonally non-computable 2-valued functions. 
By \cite{Jockusch197201} almost all oracles fail to compute a member of $\DNC_2$.
\begin{thm}\label{b92E3DHAD}
There exists a total computable $f:\twome\to\twome$ which is
\begin{enumerate}[(i)]
\item a random-preserving  oneway surjection
\item collision-resistant and  nowhere injective
\end{enumerate}
and  each  random $w\not\geqT \zj$:
\begin{itemize}
\item  does not compute any probabilistic inversion of $f$ 
\item  does not compute  any pair of $f$-siblings. 
\end{itemize}
\end{thm}\begin{proof}
Let $h$ be the hash of Lemma \ref{SkANwtV64A} and $f$ be the $h$-shuffle. 

By Corollary \ref{QZG9TROwmy}  $f$
is a random-preserving    nowhere injective surjection
and no $w\not\geqT \zj$ computes any probabilistic inversion of $f$.
By \cite{MR2258713frank, Levin2013FI}:
\begin{equation}\label{K1Dnu41wrJ}
\textrm{if $w$ is random and computes a member of $\DNC_2$ then $w\geqT \zj$}
\end{equation}
By the choice of $h$ and \eqref{K1Dnu41wrJ}, if $w\not\geqT \zj$ is random then it does not compute any pair of $f$-siblings.
In particular $f$ is oneway and collision-resistant.
\end{proof}

\subsection{Hashing by  inseparable sets}\label{epEBCzKPQM}
Given \ce   $B, C\subseteq\Nat$ with $B\cap C=\emptyset$ if $M\subseteq\Nat$  and
\[
\parb{M\supseteq B\wedga M\cap C=\emptyset}\ \vee\ 
\parb{M\supseteq C\wedga M\cap B=\emptyset}
\]
we say that $M$ is  {\em $(B,C)$-separating}.

\begin{lem}\label{SkANwtV64Ab}
If $B, C\subseteq\Nat$ are disjoint \ce sets
there is a hash $h$ such that  every pair of siblings for the $h$-shuffle 
computes a $(B,C)$-separating set.
\end{lem}\begin{proof}
Let $(\sigma_i, n_i)$ be an effective enumeration of $\twomel\times (B\cup C)$ without repetitions
and set
\[
A:=\sqbrad{\tuple{\sigma_i, n_i}}{i\in\Nat}.
\]
Consider the $A$-hash given by
\[
h(\tau):=
\begin{cases}
B(n_i)&\textrm{if $\sigma_i\preceq\tau\wedga \tau\in 2^{\tuple{\sigma_i, n_i}}$}\\[0.2cm]
\ \ 0&\textrm{if $\sigma_i\not\preceq\tau\wedga \tau\in 2^{\tuple{\sigma_i, n_i}}$}
\end{cases}
\]
so the  $h$-shuffle is given by 
$f(x)(i)=:h(x\restr_{\tuple{\sigma_i, n_i}})\otimes x(\tuple{\sigma_i, n_i})$.

Suppose that $x,z$ are $f$-siblings and let 
$\sigma$ be the least prefix of $x$ which is not a prefix of $z$. 
By the definition of $f$ and $f(x)=f(z)$ for all $n$ 
\begin{align*}
\parb{n\in B\cup C\wedga x(\tuple{\sigma,n})\neq z(\tuple{\sigma,n})} \implies & n\in B \\[0.2cm]
\parb{n\in B\cup C\wedga x(\tuple{\sigma,n})= z(\tuple{\sigma,n})} \implies & n\in C.
\end{align*}
So $M:=\sqbrad{n}{x(n)\neq z(n)}$ is  $(B,C)$-separating and $M\leqT x\oplus z$.
\end{proof}
We say that \ce sets $B, C$ are {\em computably inseparable} if there is no computable $(B,C)$-separating set.
By \cite[Theorem 11.2.5]{Odifreddi89} the sets 
\[
H_i:=\sqbrad{n}{\varphi_n(n)=i}, i<2 
\]
are computably inseparable.
Since every $(H_0, H_1)$-separating set is in $\DNC_2$, 
Lemma \ref{SkANwtV64Ab} gives an alternative proof of Theorem \ref{b92E3DHAD}. Let
\[
\textrm{$\Ss(B, C)$ denote the  class of $(B, C)$-separating sets}.
\]
Let $B, C$ be \ce computably inseparable sets so $\Ss(B, C)$ is a $\pz$ class.
By \cite[Proposition 111.6.2]{Odifreddi89} 
every \ce Turing  degree contains such $B, C$.

By \cite[Theorem 5.3]{Jockusch197201} $\Ss(B,C)$ is a negligible $\szth(B\cup C)$ class.

We say that $w$ is {\em weakly 2-random} relative to $A$ if it is not a member of any $\szth(A)$ null class.
For such reals $w$ we have $w\not\geqT A$.
\begin{thm}
For each noncomputable \ce set $A$ there exists a 
total computable random-preserving nowhere injective  $f:\twome\to\twome$ such that
\begin{enumerate}[(i)]
\item $f$ is a oneway collision-resistant surjection
\item $A$ computes an inversion of $f$ and a pair of $f$-siblings 
\end{enumerate}

and if $w$ is weakly 2-random relative to $A$ then 
\begin{itemize}
\item  $w$ does not compute any probabilistic inversion of $f$ 
\item  $w$ does not compute  any pair of $f$-siblings. 
\end{itemize}
\end{thm}\begin{proof}
By \cite[Proposition 111.6.2]{Odifreddi89}  there exist computably inseparable \ce $B, C$ with 
$A\equivT B\equivT C$.
Since $B, C$ are disjoint $A\equivT B\cup C$ so
\begin{equation}\label{C8FbRAV3pH}
\textrm{$\mathsf{S}(B,C)$ is a null $\szth(A)$ class.}
\end{equation}
Let $h$ be the $(B\cup C)$-hash of Lemma \ref{SkANwtV64Ab} 
and  $f$ be the $h$-shuffle. 

Let $w\not\geqT A$ be weakly 2-random relative to $A$ so $w\not\geqT A$.

By Corollary \ref{QZG9TROwmy}  $f$
is a random-preserving  nowhere injective surjection
and  there is no probabilistic inversion $g\leqT w$ of $f$.
By Lemma \ref{ty4nwYJGee} (i) there is an $A$-computable  inversion of $f$.

By \eqref{C8FbRAV3pH}, the choice of $h,w$ 
and Lemma \ref{SkANwtV64Ab} 
$w$ does not compute  any pair of $f$-siblings. In particular $f$ is oneway and collision resistant. 
\end{proof}
\begin{coro}
For each noncomputable \ce $A$ there exists a total computable nowhere injective surjection $f:\twome\to\twome$
such that
\begin{itemize}
\item  $f$ is oneway and collision-resistant relative to almost all oracles
\item  $f$ is not oneway and not collision-resistant relative to $A$.
\end{itemize}
\end{coro}

\section{Injective oneway functions}\label{952rgaHVWk}
Injective oneway  maps (in particular permutations) are well-studied  in computational complexity and cryptography 
\cite{onewayPermImp, onewaypermRothe}. It is therefore interesting to know if there are injective oneway real functions $f$.
It is not hard to show that such $f$ cannot be total computable \cite[Corollary 3.2]{onewayg24}.
We do not know if partial computable oneway injections exist. 
However we obtain an upper bound on their
strength: the  oracles that can probabilistically invert them. 

An interesting corollary is that, in general, it is easier to invert (even without random oracles) 
partial computable random-preserving injections than probabilistically invert
total oneway real functions. 

A {\em tree} $T$ is a $\preceq$-downward closed subset of $\twolo$. 
A real $x$ is a {\em path} of $T$ if $x\restr_n\in T$ for all $n$. 
Let $[T]$ be the class of all paths of $T$. 
Recall the notion of {\em representations} of functions from \S\ref{DZjVDbI3sf}.

\begin{lem}\label{injectiveOnClosedSet}
Suppose $f:\subseteq 2^\omega\to 2^\omega$ is partial computable and
\begin{itemize}
\item $P\leqT w$ is a tree with  $[P]\subseteq \mathsf{dom}(f)$
\item the restriction of $f$ to  $[P]\neq \emptyset$ is injective.
\end{itemize}
There is $g\leqT w$ with $g(f(x))=x$ for all $x\in [P]$. 
\end{lem}\begin{proof}
Let $\hat f$ be a computable representation of $f$ with $\abs{\hat f(\sigma)}\leq |\sigma|$ and 
\[
P_s:=P\cap 2^{\ell_s}
\hspace{0.3cm}\textrm{where}\hspace{0.3cm}
\ell_s:=\min\sqbrad{t}{\forall \sigma\in P\cap 2^t,\ \abs{\hat f(\sigma)}>s}.
\]
Since $[P]\subseteq \dom(f)$, $P\leqT w$ the family $(P_s)$ is $w$-computable. Let
\[
B^\tau=\{\sigma\in P_{|\tau|}: \tau\preceq \hat{f}(\sigma)\}\subseteq P
\] 
and $\hat{g}(\tau)$ be the longest common prefix of the strings in $B^\tau$.

Since $(P_{|\tau|})$ is a $w$-computable family of finite sets:
\begin{itemize}
\item $\hat g\leqT w$ is a representation of some $g:\subseteq\twome\to\twome$
\item $g(y)\de\iff \lim_{\tau\prec y} \abs{\hat g(\tau)}=\infty$.
\end{itemize}
Assuming  $x\in [P], f(x)=y$ we show $g(y)=x$. If  $g(y)\de$ then 
\[
x \in \bigcap_{\tau\prec y} \dbra{B^{\tau}}\subseteq \bigcap_{\tau\prec y} \dbra{\hat g(\tau)}=\{g(y)\}
\]
so $g(y)=x$.
It remains to show that $g(y)\downarrow$. 

For a contradiction suppose that $\forall s,\ \hat{g}(y\upharpoonright_s)\preceq\sigma$ for some $\sigma$ so 
\[
\forall i<2\ \ \forall \tau\prec y:\ \dbra{B^{\tau}}\cap\llbracket\sigma i\rrbracket\neq\emptyset.
\] 
Since each $\dbra{B^{\tau}}$ is closed, by compactness there exist $x_0, x_1$ with 
\[
x_i\in\bigcap_{\tau\prec y} \parB{\dbra{B^{\tau}}\cap\llbracket\sigma i\rrbracket}\subseteq [P].
\] 
Since $ [P]\subseteq\dom(f)$ for each $i<2$ we get $f(x_i)\de$ and
\[
\forall \tau\prec y\ \ \forall \theta\prec x_i:\ 
\llbracket\hat{f}(\theta)\rrbracket\cap\llbracket \tau\rrbracket\neq\emptyset
\] 
which implies $f(x_i)=y$. Since
$x_i\in\llbracket\sigma i\rrbracket$ we have $x_0\neq x_1$ which contradicts 
the hypothesis that $f$ is injective on $[P]$. 
\end{proof}

A function $p:\Nat\to\Nat$ is {\em almost everywhere dominating (a.e.d.)} if it
dominates all $q:\Nat\to\Nat$, $q\leq z$ for almost all oracles $z$. By \cite{DobrinenSimpson04}
there exists such $p\leqT\zj$.
Oracles that compute an a.e.d.\ function are called {\em a.e.d.}

This notion can be characterized in terms of relative randomness.

Let $x\leqlr y$ denote that  every $y$-random is $x$-random.
\begin{lem}[\cite{MR2336587, solomonAED}]\label{Okf25tOoye}
The following are equivalent for each $x$:
\begin{enumerate}[(i)]
\item $x$ is \aed 
\item $\zj\leqlr x$  
\item  every positive $\pzt$  class has a positive $\pz(x)$ subclass. 
\end{enumerate}
\end{lem}
Note that if $f:\subseteq\twome\to\twome$ is partial computable:
\begin{enumerate}[\hthree(a)]
\item the domain of $f$ is  a $\Pi^0_2$ class (\eg \cite[Proposition 2.1]{owrand24})
\item if $P\subseteq\dom(f)$ is $\pz(w)$ then $f(P)\in \pz(w)$ (\eg \cite[Proposition 2.2]{owrand24}).
\end{enumerate}
By \cite[Theorem 4.3]{milleryutran} if $x$ is $z$-random and $y\leq_T x$ is random then $y$ is $z$-random. 
So if $f$ is random-preserving it  preserves $z$-randomness for all $z$. 

Partial computable injections are not oneway relative to any $w\geqlr\zj$.
\begin{thm}\label{F3pmX3isl}
Let $f:\subseteq 2^\omega\to 2^\omega$ be a partial computable 
 injection and $w\geqlr \zj$ then $f$ is not oneway relative to $w$.
\end{thm}\begin{proof}
Assuming that $f$ is random-preserving by the hypothesis:
\begin{itemize}
\item $\mathsf{dom}(f)$  is a positive $\pzt$ class by (a)
\item there is a positive $\Pi^0_1(w)$ class  $P\subseteq \mathsf{dom}(f)$ by Lemma \ref{Okf25tOoye}.
\item $f(P)\in \pz(w)$ by (b).
\end{itemize}
 By Lemma \ref{injectiveOnClosedSet} let $g\leqT w$ be such that  $\forall x\in P,\ g(f(x))=x$. Then
\begin{itemize}
\item  $P$   has a $w'$-random member $x$ because it is positive
\item $f(x)$ is $w'$-random because $f$ is random-preserving.
\end{itemize}
Since $f(P)$ is $\pz(w)$ class with a $w'$-random member:
\[
\mu(f(P))>0
\hspace{0.3cm}\textrm{and}\hspace{0.3cm}
\forall y\in f(P), \ f(g(y))=y
\]
so $f$ is not oneway relative to $w$.
\end{proof}

By \S\ref{DZjVDbI3sf} there is a total computable $f:\twome\to\twome$ which is oneway relative to 
any $w\not\geq_T\zj$. Conversely it is not hard to show that no total computable $f$ is oneway relative to $\zj$,
see \cite[Theorem 2.9]{owrand24}.  

On the other hand by \cite{GreenbergDomination} there are $w\not\geqT \zj$ with $w\geqlr\zj$.
So assuming randomness-preservation, by Theorem \ref{F3pmX3isl}  partial computable 
injections are in general easier to invert than total computable many-to-one functions.

\end{document}